\documentclass[12pt]{article}

\usepackage{graphicx, amsmath, amssymb, amsthm, multicol, float, subcaption,verbatim}
\usepackage{bbm}
\usepackage[utf8x]{inputenc}
\usepackage{epstopdf}
\usepackage[dvipsnames]{xcolor}
\usepackage{algorithmic,algorithm}
\usepackage{caption}
\usepackage{subcaption}

\usepackage{cite}

\usepackage{color}   
\usepackage{hyperref}
\hypersetup{
	hidelinks,
    linktoc=all     
}

\newtheorem{definition}{Definition}[section]
\newtheorem{theorem}[definition]{Theorem}
\newtheorem{lemma}[definition]{Lemma}
\theoremstyle{definition}
\newtheorem{remark}[definition]{Remark}
\newtheorem{example}[definition]{Example}
\newtheorem{assumption}[definition]{Assumption}
\newtheorem{alg}[definition]{Algorithm}

\numberwithin{equation}{section}

\newcommand{\cl}{\operatorname{cl}}
\newcommand{\interior}{\operatorname{int}}
\newcommand{\relint}{\operatorname{relint}}
\newcommand{\conv}{\operatorname{conv}}
\newcommand{\cone}{\operatorname{cone}}

\newcommand{\diam}{\operatorname{diam}}

\newcommand{\trans}[1]{#1^{\mathsf{T}}}

\newcommand{\N}{\mathbb{N}}
\newcommand{\R}{\mathbb{R}}
\newcommand{\1}{\mathbbm{1}}
\newcommand{\X}{\mathbb{X}}
\newcommand{\cY}{\mathcal{Y}}
\newcommand{\cA}{\mathcal{A}}

\begin{document}
\title{Approximations of unbounded convex projections and unbounded convex sets}
\author{Gabriela Kováčová\thanks{University of California, Los Angeles, Department of Mathematics, Los Angeles, CA 90095, USA, kovacova@ucla.edu.} \and Birgit Rudloff\thanks{Vienna University of Economics and Business, Institute for Statistics and Mathematics, Vienna A-1020, AUT, birgit.rudloff@wu.ac.at.} 
}
\maketitle

\begin{abstract}
We consider the problem of projecting a convex set onto a subspace or, equivalently formulated,
the problem of computing a set obtained by applying a linear mapping to a convex feasible set. This includes the problem of approximating convex sets by polyhedrons. The existing literature on convex projections provides methods for bounded convex sets only, in this paper we propose a method that can handle both bounded and unbounded problems. The algorithms we propose build on the ideas of inner and outer approximation. In particular, we adapt the recently proposed methods for solving unbounded convex vector optimization problems to handle also the class of projection problems.
\end{abstract}

\section{Introduction}

In this work, we consider a convex projection problem with the focus on the unbounded case. Our interest in the unbounded  convex projection problem arises from the appearance of potentially unbounded convex projections in the implementation of the set-valued Bellman principle. The set-valued Bellman principle is a counterpart of the famous Bellman principle for non-standard problems, such as problems with multiple objectives, see e.g.~\cite{KR21}, ratios of two objectives~\cite{KRC20} or set-valued objectives, see e.g.~\cite{FR17,FRZ21}. Convex projections have also been applied to solve reachable sets of control systems, see~\cite{SZC18}.

The aim is to compute (or approximate) a set obtained by projecting elements of a convex feasible set. The problem formulation appearing in~\cite{KR21b,SZC18} is to
\begin{align}
\label{CP0}
\text{compute } \{ y \in \R^k \; \vert \;  \exists z \in \R^m : (z,y) \in S \},
\end{align}
for a given convex feasible set $S \subseteq \R^{m+k}$. Note that the problem is equivalent to computing (or approximating) a convex set, see Lemma~\ref{lemma_formulation}. 

Problem~\eqref{CP0} with a polyhedral feasible set $S$ is known as a polyhedral projection and was studied in~\cite{LW16}. The authors of~\cite{LW16} proved an equivalence between polyhedral projection, multi-objective linear programming and vector linear programming. As a consequence, a polyhedral projection~\eqref{CP0} can be solved via an associated multi-objective linear program
\begin{align}
\label{MOCP}
\text{minimize } \begin{pmatrix}
y \\ - \trans{\1} y
\end{pmatrix} \text{ w.r.t. } \leq_{\R^{k+1}_+} \text{ subject to } (z,y) \in S.
\end{align}
This inspired subsequent research also for the convex case. \cite{SZC18} considers a convex problem~\eqref{CP0} with a compact feasible set and proposes an algorithm to solve the problem. Their algorithm adapts the primal algorithm of~\cite{LRU14} applied to the multi-objective problem~\eqref{MOCP}. \cite{KR21b} directly examines the relation between convex problems~\eqref{CP0} and~\eqref{MOCP}. The equivalence between boundedness properties of the two problems is shown and their corresponding solutions are related in the bounded (and self-bounded) case, see Section~\ref{sec:prob} for precise definitions. A polyhedral approximation of a convex body through the associated multi-objective problem is considered in~\cite{LZS21}. The authors derive error bounds in dependence of the stopping criterion of the algorithm for the multi-objective problem.

The aim of this paper is to provide a solution method for the unbounded case, which has not been treated in the existing literature so far. We start, however, by revisiting the bounded case. The existing methods solve a bounded convex projection, implicitly or explicitly, through the associated multi-objective problem.  In consequence, computations are done in the image space of the multi-objective problem~\eqref{MOCP} which has a higher dimension than the image space of the projection problem~\eqref{CP0}. We propose an algorithm for solving a bounded projection problem working directly in the image space of the problem. The algorithm is based on an inner and outer approximation of the target set; these ideas have been used for solving convex vector optimization problems, see e.g.~\cite{LRU14,AUU22}, as well as for approximation of convex bodies, see e.g.~\cite{K92,B08}. The proposed algorithm can be seen as an alternative to existing methods for the bounded case that does not require an increase of the dimension. It is also used as a part of the proposed solution method in the unbounded case.

In case of unbounded convex projections, also the associated multi-objective problem is unbounded. A method for solving unbounded multi-objective convex problems has only recently been introduced in~\cite{WURKH22}. In this paper we adapt the ideas of the algorithms of~\cite{WURKH22} and propose a method for solving unbounded convex projections. The problem of approximating convex sets by polyhedrons was recently studied in~\cite{D22, DL23}, were algorithms were developed for the particular case of spectrahedra.

\section{Preliminaries}
For a set $X \subseteq \R^n$ we denote the closure of $X$, the interior of $X$, the relative interior of $X$, the convex hull of $X$ and the convex conic hull of $X$ by $\cl X, \interior X, \relint X, \conv X$ and $\cone X$, respectively. We use the convention $\cone \emptyset := \{0\}$. The \textit{recession cone} of a set $X \subseteq \R^n$ is
$$X_\infty = \{ d\in \R^n \; : \; x + \lambda d \in X \quad \forall x \in X, \lambda \geq 0 \}.$$
Elements of the recession cone are called \textit{recession directions}. A cone $X \subseteq \R^n$ is called solid if $\interior X \neq \emptyset$.

 Every \textit{polyhedron} has an \textit{H-representation} and a \textit{V-representation}. In the \textit{H-representation} a polyhedron $X \subseteq \R^n$ is given as a finite intersection of half-spaces
$$X = \bigcap\limits_{i=1}^k \{ x \in \R^n \; : \; \trans{w}_i x \leq \alpha_i \},$$
where $k \in \N, w_1 , \dots, w_k \in \R^n \setminus \{0\}$ and $\alpha_1, \dots, \alpha_k \in \R$. In the \textit{V-representation} a polyhedron $X \subseteq \R^n$ is determined by points $v_1, \dots, v_{k_v} \in \R^n$ and directions $d_1, \dots, d_{k_d} \in \R^n \setminus \{0\}$, for $k_v \in \N \setminus \{0\}$ and $k_d \in \N$,  as
$$X = \conv \{v_1, \dots, v_{k_v}\} + \cone \{ d_1, \dots, d_{k_d} \}.$$
We assume that the set of points as well as the set of directions do not contain any redundant elements. In the sequel, we denote by $\text{vert} X$ the set of points $v_1, \dots, v_{k_v} \in \R^n$ determining the V-representation of $X$. Note that when the polyhedron does not contain any line, the points $\text{vert} X$ are the vertices of $X$ and the V-representation is uniquely determined. For a polyhedron containing a line we consider an arbitrary (non-redundant) V-representation. 

We primarily use the $l_1$ (Manhattan) norm, denoted by $\Vert \cdot \Vert := \Vert \cdot \Vert_1$, to measure distances. When the $l_2$ (Euclidean) norm is used it is denoted by $\Vert \cdot \Vert_2$. A closed $\epsilon$-ball around the origin (in the $l_1$-norm) is denoted by $B_{\epsilon}$. The main reason for the choice of the $l_1$-norm is the fact (utilized within Algorithm~\ref{alg_rec}) that the associated ball is a polyhedron.
  The distance between two sets $X, Y \subseteq \R^n$ is measured as their \textit{Hausdorff distance}
$$d_H (X, Y) = \max \left\lbrace \sup\limits_{x \in X} \inf\limits_{y \in Y} \Vert x - y \Vert, \quad \sup\limits_{y \in Y} \inf\limits_{x \in X} \Vert x - y \Vert \right\rbrace,$$
where we use the convention $d_H(\emptyset, \emptyset) := 0$.
Note that $d_H (X, Y) \leq \epsilon$ for closed sets $X$ and $Y$ implies $X \subseteq Y + B_\epsilon$ and $Y \subseteq X + B_\epsilon$. The \textit{diameter} of a set $X \subseteq \R^n$ is defined as 
$$\diam (X) = \sup \{ \Vert x-y \Vert \; : \; x, y \in X \}.$$

\section{Problem formulation and solution concepts}\label{sec:prob} 
Throughout the paper we will use the problem formulation of computing a set obtained by applying a linear mapping onto a convex feasible set. This is an equivalent formulation of the convex projection problem, as Lemma~\ref{lemma_formulation} below shows. The formulation was chosen as it requires fewer pieces of notation. This section introduces the problem, its properties and appropriate solution concepts.

We will denote by $\X \subseteq \R^n$ a convex feasible set and by $A \in \R^{a \times n}$ a matrix. The problem of interest is
\begin{align}
\label{CP}
\tag{P}
\text{compute } A[\X] = \{ Ax \; : \; x \in \X\}.
\end{align}
Since the set $A[\X]$ does not need to be closed, even if the feasible set $\X$ is closed, we will throughout the paper work with the closed image set
\begin{align*}
\cA = \cl A[\X].
\end{align*}

\begin{lemma}
\label{lemma_formulation}
Problem~\eqref{CP} and problem~\eqref{CP0} are equivalent.
\end{lemma}
\begin{proof}
Problem~\eqref{CP0} can be formulated in the form of problem~\eqref{CP} for the particular choice of $a = k, n = m+k, \X = S$ and $A = (0_{k \times m}, I_{k \times k})$.
Problem~\eqref{CP} can be formulated in the form of problem~\eqref{CP0} for the particular choice of $m=n, k = a$ and $S = \{(x,y) \in \R^{n+a} \; \vert \; x \in \X, y = Ax\}$.
\end{proof}

The following summarizes the assumptions made throughout this paper.
\begin{assumption}
\label{Ass}
 $\X \subseteq \R^n$ is a convex set with a non-empty interior, $\interior \X \neq \emptyset$.
 $A \in \R^{a \times n}$ is a matrix. They jointly determine the problem~\eqref{CP}.

\end{assumption}

The boundedness properties of problem~\eqref{CP} determine which solution concepts and solution methods are appropriate, as one can already see in~\cite{KR21b, SZC18} and is also discussed below. 
We now provide a definition of a bounded problem and shortly discuss afterwards various notions of the more general concept of self-boundedness that one can encounter in the literature. Then, we recall the solution concept for bounded problems and propose a solution concept for the unbounded case.
\begin{definition}
Problem~\eqref{CP} is \textbf{bounded} if $\cA$ is a bounded set. Otherwise, it is \textbf{unbounded}. 
\end{definition}

Consider the following notions of \textit{self-boundedness} for a closed convex set $\cA \subseteq \R^a$:
\begin{enumerate}
\item[(S1)] $\cA \subseteq \{a^0\} + \cA_{\infty}$ for some $a^0 \in \R^a$.
\item[(S2)] $\cA \subseteq \conv \{a^1, \dots, a^k\} + \cA_{\infty}$ for some $k \in \N, a^1, \dots, a^k \in \R^a$.
\item[(S3)] $\sup\limits_{x \in \cA} \inf\limits_{y \in \cA_{\infty}} \Vert x - y \Vert < \infty$.
\end{enumerate}
The notion (S1) was introduced in~\cite{U18} for (an upper image of) a vector optimization problem. (S2) was proposed in~\cite{KR21b} as a generalization of (S1) appropriate for a convex projection problem. (S3) appears in~\cite{D22} as a notion appropriate for general (line-free) convex sets, the concept originates from~\cite{NR95}. These three properties are closely related, as the next lemma summarizes.
\begin{lemma}
Consider a closed convex set $\cA \subseteq \R^a$.
It holds $(S1) \Rightarrow (S2) \Leftrightarrow (S3)$. If the recession cone $\cA_{\infty}$ is solid, then all three notions are equivalent.
\end{lemma}
\begin{proof}
$(S1) \Rightarrow (S2)$ is obvious. 
Let us show $(S2) \Leftrightarrow (S3)$:
\begin{itemize}
\item Assume $(S2)$ and take any $x \in \cA$. Then $x = \sum \alpha^i a^i + y$ for some $y \in \cA_{\infty}$ and $\alpha^i \geq 0$ with $\sum \alpha^i = 1$. We get $\Vert x - y \Vert = \Vert \sum \alpha^i a^i \Vert \leq \max\limits_{i = 1, \dots, k} \Vert a^i \Vert$, which proves $(S3)$.
\item Assume $(S3)$ and set $K:= \sup\limits_{x \in \cA} \inf\limits_{y \in \cA_{\infty}} \Vert x - y \Vert < \infty$. This means that for every $x \in \cA$ there exists $y \in \cA_{\infty}$ such that $x-y \in B_K$, therefore $\cA \subseteq \cA_{\infty} + B_K$. Since the polyhedral $\ell_1$ ball $B_K$ is a convex hull of finitely many points, this proves $(S2)$.   
\end{itemize}
\cite{D22} showed $(S1) \Leftrightarrow (S3)$ for a set with a solid $\cA_{\infty}$. 
\end{proof}

Finally, we discuss solution concepts for the considered problem. What we understand under \textit{solving~\eqref{CP}} is approximating the set $\cA$ (and hence also $A[\X]$) in some (appropriate) sense. The following definition of a solution that is appropriate for a bounded problem is taken from~\cite{KR21b}.
\begin{definition}
A non-empty finite set $\bar{X}\subseteq \X$ is a \textbf{finite $\epsilon$-solution} of a bounded problem~\eqref{CP} if it holds 
$$\conv A[\bar{X}] + B_{\epsilon} \supseteq \cA.$$
\end{definition}
\cite{KR21b} shows that such a finite $\epsilon$-solution exists for an arbitrary tolerance $\epsilon > 0$ for a bounded convex projection problem. An approach for finding such solutions through solving the associated multi-objective problems is also provided -- we shortly recapped this in the introduction. 
However, this solution concept is clearly not appropriate once we drop the assumption of boundedness -- an unbounded set $\cA$ cannot be reasonably approximated by a bounded polyhedron $\conv A[\bar{X}] + B_{\epsilon}$. This is true even for unbounded problems that are self-bounded.

For this reason, we propose a more general notion of approximate solution here that is appropriate for bounded as well as unbounded problems (regardless of being self-bounded or not). It is based on the notion of a finite $(\epsilon, \delta)$-solution of a convex vector optimization problem considered in~\cite{WURKH22}. The main idea is for the solution to contain, alongside a set of feasible points, also a set of directions to approximate the recession cone of the image set. 

\begin{definition}
\label{def_sol_unbound}
A pair $(\bar{X}, \cY)$ is a \textbf{finite $(\epsilon, \delta)$-solution} of problem~\eqref{CP} if  $\emptyset \neq \bar{X} \subseteq \X$ and $\cY \subseteq \R^a$ are finite sets that satisfy
\begin{enumerate}
\item $d_H (\cA_{\infty} \cap B_1, \cone \cY \cap B_1) \leq \delta$,
\item $\cA_{\infty} \subseteq \cone \cY$, and
\item $\conv A[\bar{X}] + \cone \cY + B_{\epsilon} \supseteq \cA$.
\end{enumerate}
\end{definition}

The above notion of a finite $(\epsilon, \delta)$-solution (as well as the counterpart in~\cite{WURKH22} for convex vector optimization problems) is closely related to the notion of a (polyhedral) outer approximation of a set introduced in~\cite{D22}. We state this definition here as well with one modification -- throughout this work we use the Manhattan norm instead of the Euclidean norm.
\begin{definition}{\cite[Definition 4.2]{D22}}
\label{def_outer_app}
Given a nonempty, closed convex and line-free set $C \subseteq R^a$, a line-free polyhedron $P$ is called an $(\epsilon, \delta)$-outer approximation of $C$ if
\begin{enumerate}
\item $d_H (P_{\infty} \cap B_1, C_{\infty} \cap B_1) \leq \delta$,
\item $P \supseteq C$,
\item $\sup\limits_{p \in \text{vert }P} \inf\limits_{c \in C} \Vert p - c \Vert \leq \epsilon$. 
\end{enumerate}
\end{definition}

The following relation holds between Definitions~\ref{def_sol_unbound} and~\ref{def_outer_app}: 
A finite $(\epsilon, \delta)$-solution $(\bar{X}, \cY)$ of~\eqref{CP} generates a polyhedron $\conv A[\bar{X}] + \cone \cY + B_{\epsilon}$ that is an $(\epsilon, \delta)$-outer approximation of the set $\cA$. For a bounded problem, a finite $\epsilon$-solution $\bar{X}$ yields a finite $(\epsilon, 0)$-solution $(\bar{X}, \emptyset)$ in the sense of Definition~\ref{def_sol_unbound} and it also generates a polyhedron $\conv A[\bar{X}] + B_{\epsilon}$ that is an $(\epsilon, 0)$-outer approximation of the bounded set $\cA$.

\section{Scalarizations of~(\ref{CP})}
The previous works provide a way to solve a (bounded) problem~\eqref{CP} through the associated multi-objective problem. These approaches, such as~\cite[Algorithm 3.3]{SZC18}, work with scalarizations of the associated multi-objective problem.
In this work we will construct algorithms to solve the problem~\eqref{CP} directly, without (implicitly or explicitly) transforming it into a multi-objective problem. Therefore, we will need scalarizations of problem~\eqref{CP}. This sections introduces three scalarizations -- the weighted sum, the Pascoletti-Serafini and the norm minimization -- and provides some related results. Parallel results for scalarizations of a convex vector optimization problem can be found in~\cite{LRU14, AUU22}.

The \textbf{weighted sum scalarization} of~\eqref{CP}, given a weight $w \in\R^a\setminus\{0\}$, is
\begin{align}
\tag{$P_1(w)$}
\label{P1}
\min\limits_{x \in \X} \trans{w} Ax.
\end{align}
\begin{lemma} \label{lemma_P1}
If~\eqref{P1} is bounded and $x^w$ is the optimal solution, then $\{ y \in \R^a \; : \; \trans{w} y \geq \trans{w} Ax^w \} \supseteq \mathcal{A}$.
\end{lemma}

The \textbf{Pascoletti-Serafini scalarization} of~\eqref{CP} with point $v \in \R^a$ and direction $d \in \R^a \setminus \{0\}$ is
\begin{align}
\tag{$P_2(v, d)$}
\label{P2}
\max\limits_{(x,\alpha) \in \X\times\R} \alpha  \quad \text{ s.t. } \quad Ax = v + \alpha d.
\end{align}

Its dual problem is
\begin{align}
\tag{$D_2(v, d)$}
\label{D2}
\min\limits_{\lambda \in \R^{a}} \left\lbrace -\trans{\lambda}v + \sup_{x \in \X} \{\trans{\lambda} Ax\}\right\rbrace \quad \text{ s.t. } \quad \trans{\lambda} d = 1.
\end{align}

\begin{remark}
\label{rem:P2-duality}
Strong duality between \eqref{P2} and its dual \eqref{D2} does not hold in general. As a counterexample one can consider $\X = \{ x \in \R^2 \vert x_1 > 0, x_2 \geq 0 \} \cup \{ x \in \R^2 \vert x_1 = 0, 0 \leq x_2 \leq 1 \}$, the matrix $A$ being the identity, the point $v = (0,0)$ and the direction $d = (0,1)$. Then, the primal problem \eqref{P2} is bounded with optimal value $1$, but the dual problem \eqref{D2} is unbounded. In general, the existence of a pair $(x,\alpha)$ strictly feasible for the primal problem is sufficient for strong duality, see~\cite[Section 5.2.3]{BV04}. Therefore, strong duality holds when the point $v$ is chosen from the relative interior of the set $A[\X]$.
\end{remark}

\begin{lemma}\label{lemma_P2_hyperplane}
Let $v \in \relint A[\X]$ and let problem~\eqref{P2} be bounded. Assume that $(x^*, \alpha^*)$ and $\lambda^*$ are optimal solutions of~\eqref{P2} and~\eqref{D2}, respectively. Then it holds $\{ y \in \R^a \; : \; \trans{(\lambda^*)} y \leq \trans{(\lambda^*)} A x^* \} \supseteq \mathcal{A} \supseteq A[\X]$.
\end{lemma}
\begin{proof}
Strong duality and feasibility give
\begin{align*}
\sup_{x \in \X} \{\trans{(\lambda^*)} Ax\} = \trans{(\lambda^*)} v + \alpha^* = \trans{(\lambda^*)} (Ax^* - \alpha^* d) + \alpha^* = \trans{(\lambda^*)} Ax^*.
\end{align*}
This implies $\trans{(\lambda^*)} Ax \leq \trans{(\lambda^*)} Ax^*$ for all $x \in \X$ and the result follows.
\end{proof}

\begin{lemma}\label{lemma_P2_rec}
If the problem~\eqref{P2} is unbounded, then $d$ is a recession direction of $\mathcal{A}$.
\end{lemma}
\begin{proof}
Since $\mathcal{A}$ is a closed convex set, to prove that $d \in \mathcal{A}_{\infty}$ it suffices to show that for some $v^0 \in \cA$ it holds
$$v^0 + \alpha d \in \mathcal{A} \quad \forall \alpha \geq 0.$$

Let us first find a point $v^0 \in \cA$: Unboundedness of the problem~\eqref{P2} implies that there is a pair of $x^0 \in \X$ and $\alpha^0 \geq 0$ feasible for the problem, so $v^0 :=Ax^0= v + \alpha^0 d \in A[\X] \subseteq \mathcal{A}$. 

Now let $\alpha \geq 0$. Since~\eqref{P2} is unbounded, there exists a feasible $\alpha^1 \geq \alpha^0 + \alpha$ with a corresponding $x^1 \in \X$, which gives us $v + \alpha^1 d = Ax^1 \in A[\X]$. Since $v^0 + (\alpha^1 - \alpha^0) d = v + \alpha^1 d$, convexity of the set $A[\X]$ implies $v^0 + \alpha d \in A[\X] \subseteq \mathcal{A}$.
\end{proof}

\begin{lemma} \label{lem:P2-feasible}
If $v \in A[\X]$, then~\eqref{P2} is feasible. 
\end{lemma}

The \textbf{norm minimization scalarization} of~\eqref{CP} given a point $v \in \R^a$ is
\begin{align}
\tag{$P_3(v)$}
\label{P3}
\min\limits_{(x,z) \in\X\times \R^a} \Vert z \Vert_{2} \quad \text{ s.t. } \quad Ax = v+z.
\end{align}
Its dual problem is
\begin{align}
\tag{$D_3(v)$}
\label{D3}
\max\limits_{\lambda \in \R^a}  \left\lbrace - \trans{\lambda} v + \inf\limits_{x \in \X} \trans{\lambda} Ax \right\rbrace \quad \text{ s.t. } \quad \Vert \lambda \Vert_{2} \leq 1.
\end{align}
To derive this consider the dual function
\begin{align*}
q(\lambda) = \inf\limits_{x \in \X, z \in \R^a} L(x, z, \lambda) = -\trans{\lambda} v + \inf\limits_{x \in \X} \trans{\lambda} Ax + \inf\limits_{z \in \R^a} \{ \Vert z \Vert_{2} - \trans{\lambda} z \},
\end{align*}
where the last term corresponds to the conjugate of $\Vert z \Vert_{2}$, see~\cite[Example 3.26]{BV04}.

The Norm minimization scalarization was used by~\cite{AUU22} in an algorithm for solving convex vector optimization problems. Unlike in the rest of the paper, we used the $\ell_2$ norm here. Strong duality, boundedness and existence of a supporting hyperplane derived below hold for the scalarization with all choices of a norm. Using the $\ell_2$ norm helps to avoid possible problems in the case of non-uniqueness of an optimal solution under the $\ell_1$ norm, even though in numerical experiments both norms worked equally well. The scalarization will be utilized in Algorithm~\ref{alg_bound}.

\begin{lemma} \label{lemma_P3_fb}
Problem~\eqref{P3} is feasible and bounded for any point $v \in \R^a$. Furthermore, there is no duality gap between problems~\eqref{P3} and~\eqref{D3}.
\end{lemma}
\begin{proof}
Consider an arbitrary $\bar{x} \in \interior \; \X$ and the corresponding $\bar{z} = A \bar{x} - v$. The pair $(\bar{x}, \bar{z})$ is strictly feasible for the problem. The objective function is bounded from below by $0$. This yields feasibility, boundedness and strong duality.
\end{proof}

\begin{lemma} \label{lemma_P3}
Assume that $(x^*, z^*)$ and $\lambda^*$ are optimal solutions of~\eqref{P3} and~\eqref{D3}, respectively. Then it holds $\{ y \in \R^a \; : \; \trans{(\lambda^*)} y \geq \trans{(\lambda^*)} Ax^* \} \supseteq \mathcal{A} \supseteq A[\X]$.
\end{lemma}
\begin{proof}
From strong duality we have
\begin{align*}
\Vert z^* \Vert_{2} = - \trans{(\lambda^*)} v + \inf_{x \in \X} \trans{(\lambda^*)} Ax = - \trans{(\lambda^*)} (Ax^* - z^*) + \inf_{x \in \X} \trans{(\lambda^*)} Ax.
\end{align*}
We rearrange this expression and add the term $0 = \inf\limits_{z \in \R^a} \{ \Vert z \Vert_{2} - \trans{(\lambda^*)} z \}$, which originates from the Fenchel conjugate of the norm and dual feasibility, see~\cite[Example 3.26]{BV04}. We obtain
\begin{align*}
\trans{(\lambda^*)} Ax^* - \inf_{x \in \X} \trans{(\lambda^*)} Ax = \inf\limits_{z \in \R^a} \{ \Vert z \Vert_{2} - \trans{(\lambda^*)} z \} - (\Vert z^* \Vert_{2} - \trans{(\lambda^*)} z^*) \leq 0.
\end{align*}
Therefore, for all $x \in \X$ it holds $\trans{(\lambda^*)} Ax^* \leq \trans{(\lambda^*)} Ax$.
\end{proof}

\begin{remark}
\label{remark_near1}
Above we derived half-spaces containing the set $\cA$ based on optimal solutions of the three scalarizations and their duals. However, an optimal solution need not be available -- it might not even exist. In practice, these scalarizations will be solved by some available solver up to the computational precision of the solver. Let us restate the results of Lemmas~\ref{lemma_P1},~\ref{lemma_P2_hyperplane} and~\ref{lemma_P3} for near-optimal solutions. Here $\varepsilon$ denotes the precision.

Firstly, consider a near-optimal solution $x^\varepsilon$ of the weighted sum scalarization~\eqref{P1}. One can show  that it holds
\begin{align*}
\cA \subseteq \{ y \in \R^a \; : \; \trans{w} y \geq \trans{w} Ax^\varepsilon - \varepsilon \}.
\end{align*}

Secondly, consider near-optimal solutions $(x^\varepsilon, \alpha^\varepsilon)$ and $\lambda^\varepsilon$ of the Pascoletti-Serafini scalarization~\eqref{P2} and its dual~\eqref{D2}. One can show  that it holds
\begin{align*}
\cA \subseteq \{ y \in \R^a \; : \; \trans{(\lambda^\varepsilon)} y \leq \trans{(\lambda^\varepsilon)} A x^\varepsilon + \varepsilon \}.
\end{align*}

Thirdly, consider near-optimal solutions $(x^\varepsilon, z^\varepsilon)$ and $\lambda^\varepsilon$ of the norm-minimization scalarization~\eqref{P3} and its dual~\eqref{D3}. One can show that it holds
\begin{align*}
\cA \subseteq \{ y \in \R^a \; : \; \trans{(\lambda^\varepsilon)} y \geq \trans{(\lambda^\varepsilon)} A x^\varepsilon - \varepsilon \}.
\end{align*}
Hence, the halfspaces are slightly shifted if only near-optimal solutions of the scalarizations are computed. They are, however, not tilted. This will in particular be important as the recession directions searched within the algorithm proposed in Section~\ref{sec:unb} are not affected by the use of optimal or near-optimal solutions. For more details on the effect of near optimal solutions on the proposed algorithms see Remarks~\ref{remark_near2} and~\ref{Rem:near}. 
\end{remark}

\section{Algorithm for a bounded problem~(\ref{CP})}
First, we address the case of a bounded problem~\eqref{CP}. Such problems can be solved via the associated multi-objective problem by the already available algorithm \cite[Algorithm 3.3]{SZC18} or results of~\cite{KR21b}. However, doing so runs an algorithm in a $(a+1)$-dimensional space. Here we propose an algorithm that works directly in the space $\R^a$. Furthermore, it will be used as part of the algorithm for unbounded problems proposed in Section~\ref{sec:unb}.

Our algorithm is based on the idea of finding an outer approximation of the set and iteratively improving the approximation till a desired tolerance level is achieved. As such, Algorithm~\ref{alg_bound} is similar to outer approximation schemes for convex bodies, see~\cite{B08} for a survey, and to the (primal) algorithms of~\cite{LRU14,AUU22} for bounded convex vector optimization problems (CVOPs).  Unlike a CVOP, problem~\eqref{CP} does not involve an ordering cone. Therefore, to obtain an initial outer approximation of the set $\cA$, we use the weighted sum scalarizations for standard basis vectors as well as their negatives. This will provide a box-type initial approximation. Afterwards we improve the approximation by applying the norm minimizing scalarization on the vertices of the outer approximation.

\begin{alg}{\label{alg_bound}}
Algorithm for solving a bounded problem~\eqref{CP} satisfying Assumption~\ref{Ass}.
\begin{description}
\item[Input] ~\\
Bounded problem~\eqref{CP} with a feasible set $\X$ and a matrix $A$; tolerance $\epsilon > 0$. 

\item[Initialization] ~		
\begin{enumerate}
	\item Set $\bar{X} := \emptyset, \mathcal{A}_0 := \R^a$, $V^\epsilon := \emptyset$ and $H := \emptyset$
	\item For $w = e^1, \dots, e^a, -e^1, \dots, -e^a$ solve~\ref{P1} and obtain an optimal solution $x^w$.
	\begin{itemize}
	\item Set $\bar{X} := \bar{X} \cup \{x^w\}$ and
	$\mathcal{A}_0 := \mathcal{A}_0 \cap \{ y \in \R^a \; : \; \trans{w} y \geq \trans{w} Ax^w \}.$
	\end{itemize}	  
\end{enumerate}

\item[Iteration] ~	
\begin{enumerate}\addtocounter{enumi}{2}
	\item While $H \neq \R^a$
	\begin{enumerate}
	\item Set $H := \R^a$
	\item For each $v \in (\text{vert } \mathcal{A}_0) \setminus V^\epsilon$
	\begin{itemize}
	\item Solve problems~\ref{P3} and~\ref{D3} to obtain optimal solutions $(x^v, z^v)$ and  $\lambda^v$. 
	\item Update $\bar{X} := \bar{X} \cup \{x^v\}$
	\item If $\Vert z^v \Vert > \epsilon$, then update $H :=  H \cap \{ y \in \R^a \; : \; \trans{(\lambda^v)} y \geq \trans{(\lambda^v)} Ax^v  \}$. \\ Else $V^\epsilon := V^\epsilon \cup \{v\}$.
	\end{itemize}
	\item Update $\mathcal{A}_0 := \mathcal{A}_0 \cap H$
	\end{enumerate}
\end{enumerate}

\item[Output] ~\\
		Set of feasible points $\bar{X}$ and outer approximation $\cA_0$ of $\cA$.
\end{description}
\end{alg}

\begin{remark}
\label{remark_initialization}
Alternatively, we could initialize the algorithm with solving the weighted sum scalarization~\eqref{P1} for $w = e^1, \dots, e^a, -\1$ in Step 2. This would also provide a bounded set $\cA_0$. We compare the two initializations in Example~\ref{ex1}, where the alternative version requires fewer operations (optimizations and evaluations of a polyhedron). This is, however, not true for all problems.
\end{remark}

In the following we prove some results about the behaviour of Algorithm~\ref{alg_bound}. These will be afterwards used to prove the correctness of the proposed algorithm.

\begin{lemma}\label{lemma_alg_bound}
Let the problem~\eqref{CP} be (feasible and) bounded. The following holds for Algorithm~\ref{alg_bound}:
\begin{enumerate}
\item Each scalar problem~\eqref{P1} considered in Step 2 is bounded.
\item After Initialization (Steps 1 and 2), the set $\mathcal{A}_0$ is a bounded convex polyhedron that contains the set $\mathcal{A}$.
\item At each iteration, the set $\mathcal{A}_0$ is a (bounded) convex polyhedral superset of $\mathcal{A}$.
\end{enumerate}
\end{lemma}
\begin{proof}
\begin{enumerate}
\item For a bounded problem~\eqref{CP} it holds $A[\X] \subseteq B_K$, respectively $\mathcal{A} \subseteq B_K$, for some $K$. This implies that scalarization~\eqref{P1} is bounded  for all weights $w \in \R^a \setminus \{0\}$.
\item The set $\mathcal{A}_0$ is initialized as $\R^a$ and throughout Step 2 it is updated via intersections of halfspaces, which contain the set $\mathcal{A}$, see Lemma~\ref{lemma_P1}. Therefore, $\mathcal{A}_0$ is a convex polyhedral superset of $\mathcal{A}$. Since the set $\mathcal{A}_0$ consists of intersections of halfspaces of the form $\{ y \in \R^a \; : \;  y_i \geq \overline{\alpha}_i \}$ and $\{ y \in \R^a \; : \;  y_i \leq \underline{\alpha}_i \}$ for some $\overline{\alpha}_i,\underline{\alpha}_i\in\R$ and for $i = 1, \dots, a$, it is bounded. As it contains the nonempty set $\mathcal{A}$, it cannot be empty. 
\item Within each iteration, set $\mathcal{A}_0$ is updated through intersections with half-spaces that contain the set $\mathcal{A}$, see Lemmas~\ref{lemma_P3_fb} and~\ref{lemma_P3}.
\end{enumerate}
\end{proof}

\begin{theorem}
If Algorithm~\ref{alg_bound} terminated successfully for a bounded problem~\eqref{CP}, then $\bar{X}$ is a finite $\epsilon$-solution of~\eqref{CP}. Furthermore, $\conv A[\bar{X}] + B_{\epsilon}$ (as well as $\mathcal{A}_0$) is an $(\epsilon, 0)$-outer approximation of the set $\mathcal{A}$.
\end{theorem}
\begin{proof}
Assume that Algorithm~\ref{alg_bound} terminated successfully for a bounded problem~\eqref{CP}. According to Lemma~\ref{lemma_alg_bound} and convexity of $\mathcal{A}$ it holds
$$\conv A[\bar{X}] \subseteq \mathcal{A} \subseteq \mathcal{A}_0.$$
The algorithm terminates when
\begin{align*}
\forall v \in \text{vert }\mathcal{A}_0 \quad \exists \bar{x} \in \bar{X}:  \quad \Vert v - A\bar{x} \Vert \leq \epsilon.
\end{align*}
Convexity of the involved polyhedrons and convexity of the norm give
\begin{align*}
d_H (\mathcal{A}_0, \conv A[\bar{X}]) \leq \epsilon,
\end{align*}
which implies 
$$\conv A[\bar{X}] \subseteq \mathcal{A} \subseteq \mathcal{A}_0 \subseteq \conv A[\bar{X}] + B_{\epsilon}.$$
\end{proof}

\begin{remark}
\label{remark_near2}
Within Algorithm~\ref{alg_bound} we assume that the optimal solutions of the two scalarizations are found. As stated in Remark~\ref{remark_near1}, this will not be the case when the scalarizations are solved numerically (even if optimal solutions exist). Instead, near-optimal solutions will be found, up to a solver accuracy $\varepsilon$. This will results in the cutting half-spaces being shifted (however not tilted) by $\varepsilon$, see Remark~\ref{remark_near1}. One could incorporate this
quantity into the implementation of the algorithm, but usually $\varepsilon \ll \epsilon$.
\end{remark}

\begin{remark}
The correctness of Algorithm~\ref{alg_bound} is proven under the assumption that it terminates, similar to the results for the Benson-type algorithms \cite[Algorithm 1]{LRU14} and \cite[Algorithm 1]{AUU22} for bounded CVOPs. Finite termination of all these algorithms remains an open question. The authors of \cite{AUU22} proved termination in a finite number of iteration for a modification of \cite[Algorithm 2]{AUU22}. 
If a guarantee of a finite termination is desired for a bounded convex projection, we could apply \cite[Algorithm 2]{AUU22} to a counterpart of problem~\eqref{MOCP}. However, this would come at the cost of solving the problem in $(a+1)$-dimensional space and a difficulty in extending the approach to the unbounded case.
\end{remark}

\section{Algorithm for a general problem~(\ref{CP})}
\label{sec:unb}
Now we address problem~\eqref{CP} without the assumption of boundedness. Our aim is again to formulate an algorithm for solving this problem -- that is, for finding a finite $(\epsilon, \delta)$-solution $(\bar{X}, \cY)$ of~\eqref{CP}. Our approach is related in spirit to the method for solving a (possibly unbounded) convex vector optimization problem from~\cite{WURKH22}. 

Below we introduce two algorithms: Algorithm~\ref{alg_rec} searches for directions $\cY$ that generate an outer approximation of the recession cone $\cA_{\infty}$ up to a desired tolerance $\delta$. Algorithm~\ref{alg_unbound} combines Algorithm~\ref{alg_rec} (as an initialization phase) with the iterations of Algorithm~\ref{alg_bound} to also find the set of feasible points $\bar{X}$, which together with $\cY$ provide a finite $(\epsilon, \delta)$-solution of~\eqref{CP}. 

Algorithm~\ref{alg_rec} is based on~\cite[Algorithm 5.2]{WURKH22}. However, unlike a CVOP our problem does not contain an ordering cone, therefore no recession direction is known in advance. After initialization, our algorithm iteratively visits directions from the recession cone of the current outer approximation while searching for a recession direction of the set $\cA$. If a recession direction is found, it serves as an initial inner approximation of $\cA_{\infty}$, afterwards the inner approximation can be used to iteratively improve  the outer one. However, a recession direction need not be found in finitely many steps (think of $\cA_{\infty}$ that is not solid). Then the algorithm terminates once the outer approximation becomes 'thin' enough. In both cases, a finite $(\epsilon, \delta)$-solution is found.

\begin{alg}{\label{alg_rec}}
Algorithm for approximating the recession cone of $\cA$. 
\begin{description}
\item[Input] ~\\
Feasible set $\X$ and matrix $A$; tolerance $\delta > 0$, point $v \in \relint A[\X]$. 

\item[Initialization] ~		
\begin{enumerate}
	\item Set $\bar{X} := \emptyset, \mathcal{A}_0 := \R^a, \mathcal{Y}_{in} := \emptyset, \Delta := \emptyset$ 
	\item Check boundedness: \\ For $w = e^1, \dots, e^a, -e^1, \dots, -e^a$ solve problem~\ref{P1}
	\begin{itemize}
	\item If~\ref{P1} is bounded and $x^w$ is an optimal solution, then update $\bar{X} := \bar{X} \cup \{x^w\}$ and
	$\mathcal{A}_0 := \mathcal{A}_0 \cap \{ y \in \R^a \; : \; \trans{w} y \geq \trans{w} Ax^w \}.$
	\end{itemize}		
	\item Compute $\mathcal{Y}_{out} := \text{vert } ((\mathcal{A}_0)_{\infty} \cap B_1 ) \setminus \{0\}$.	  
\end{enumerate}

\item[Iteration] ~	
\begin{enumerate}\addtocounter{enumi}{3}
	\item Search for a recession direction: \\
	While $\mathcal{Y}_{in} = \emptyset$ and $\diam (\cY_{out}) > \delta$ 
	\begin{enumerate}
	\item Set $H := \R^a$
	\item For each $d \in \mathcal{Y}_{out} \cup \left\lbrace \frac{\sum_{d' \in \cY_{out}} d'}{\Vert \sum_{d' \in \cY_{out}} d' \Vert} \right\rbrace$ solve problem~\ref{P2}
	\begin{itemize}
	\item If~\ref{P2} is unbounded, then update $\mathcal{Y}_{in} := \mathcal{Y}_{in} \cup \{d\}$. 
	\item If $(x^d, \alpha^d)$ and $\lambda^d$ are optimal solutions of~\ref{P2} and~\ref{D2}, then update $\bar{X} := \bar{X} \cup \{x^d\}$ and $H := H \cap \{ y \in \R^a \; : \; \trans{(\lambda^d)} y \leq \trans{(\lambda^d)} A x^d \}$
	\end{itemize}
	\item Update $\mathcal{A}_0 := \mathcal{A}_0 \cap H$ and $\mathcal{Y}_{out} := \text{vert } ((\mathcal{A}_0)_{\infty} \cap B_1  ) \setminus \{0\}$
	\end{enumerate}
	
	\item While $\mathcal{Y}_{out} \setminus (\mathcal{Y}_{in} \cup \Delta) \neq \emptyset$ and $\diam (\cY_{out}) > \delta$
	\begin{enumerate}
	\item Set $H := \R^a$
	\item For each $d \in \mathcal{Y}_{out} \setminus (\mathcal{Y}_{in} \cup \Delta)$ find $r^d = \arg \min \{ \Vert d - r \Vert \; : \; r \in \mathcal{Y}_{in} \}$
	\begin{itemize}
	\item If $\Vert d - r^d \Vert \leq \delta$, then update $\Delta := \Delta \cup \{d\}$
	\item If $\Vert d - r^d \Vert > \delta$, then select $\beta \in (0,1)$ and solve problem $P_2 (v, d^{\beta})$ for $d^{\beta} = \frac{\beta d + (1-\beta) r^d}{\Vert \beta d + (1-\beta) r^d \Vert}$
		\begin{itemize}
		\item If $P_2 (v, d^{\beta})$ is unbounded, then update $\mathcal{Y}_{in} := \mathcal{Y}_{in} \cup \{d^{\beta}\}$
		\item If $(x^\beta, \alpha^\beta)$ and $\lambda^\beta$ are optimal solutions of $P_2 (v, d^{\beta})$ and $D_2 (v, d^{\beta})$, then update $\bar{X} := \bar{X} \cup \{x^\beta\}$ and $H := H \cap \{ y \in \R^a \; : \; \trans{(\lambda^\beta)} y \leq \trans{(\lambda^\beta)} A x^\beta \}$
		\end{itemize}
	\end{itemize}
	\item Update $\mathcal{A}_0 := \mathcal{A}_0 \cap H$ and $\mathcal{Y}_{out} := \text{vert } ((\mathcal{A}_0)_{\infty} \cap B_1  ) \setminus \{0\}$
	\end{enumerate}
\end{enumerate}

\item[Output] ~\\
		Set of feasible points $\bar{X}$ and outer approximation $\mathcal{A}_0$ of $\mathcal{A}$.\\ Inner and outer approximations $\mathcal{Y}_{in}$ and $\mathcal{Y}_{out}$ of $\mathcal{A}_{\infty}$.
\end{description}
\end{alg}

\begin{remark}
\label{remark_input_point}
In Algorithm~\ref{alg_rec}, the point $v$ from the relative interior of $A[\X]$ is treated as an input parameter. Such point can be constructed by performing a feasibility check by means of a phase I stage of an interior-point method, see~\cite[Section 11.4]{BV04}.
\end{remark}

The following lemma provides results that help to prove the correctness of Algorithm~\ref{alg_rec} in Theorem~\ref{thm_unbound_alg} below.

\begin{lemma}\label{lemma_alg_rec}
Assume that $v \in \relint A[\X]$. The following holds:
\begin{enumerate}
\item Each problem~\eqref{P2} solved in Algorithm~\ref{alg_rec} is feasible and strong duality holds for it.
\item In each step of Algorithm~\ref{alg_rec}, the set $\mathcal{Y}_{in}$  (if non-empty) contains only normalized recession directions of $\mathcal{A}$.
\item In each step of Algorithm~\ref{alg_rec}, the set $\mathcal{A}_0$ is a convex polyhedral superset of $\cA$.
\item If the problem~\eqref{CP} is bounded, then Algorithm~\ref{alg_rec} terminates with $\mathcal{Y}_{in} =  \mathcal{Y}_{out} =\emptyset$.
\end{enumerate}
\end{lemma}
\begin{proof}
\begin{enumerate}
\item This follows from Lemma~\ref{lem:P2-feasible} and Remark~\ref{rem:P2-duality}. 
\item A direction $d$ is added to the set $\mathcal{Y}_{in}$ only when the problem~\ref{P2} is unbounded, all these directions are normalized. The claim follows from Lemma~\ref{lemma_P2_rec}.
\item The set $\mathcal{A}_0$ is initialized as $\R^a$. It is updated through intersections with supporting halfspaces of $\mathcal{A}$, see Lemmas~\ref{lemma_P1} and~\ref{lemma_P2_hyperplane}.
\item For a bounded problem all weighted sum scalarizations solved in Step 2 are bounded. Therefore, the set $\cA_0$ is bounded after Step 2 and $\mathcal{Y}_{out} =\emptyset$ is obtained in Step 3. No iterations of Step 4  or Step 5 are run.
\end{enumerate}
\end{proof}

\begin{theorem}
\label{thm_rec}
If Algorithm~\ref{alg_rec} terminated successfully for problem~\eqref{CP}, then it holds
\begin{align}
\label{eq1}
\cone \mathcal{Y}_{in} \subseteq \cA_{\infty} \subseteq \cone \mathcal{Y}_{out}.
\end{align}
and
\begin{align}
\label{eq3}
d_H (\cone \mathcal{Y}_{out} \cap B_1, \cA_{\infty} \cap B_1) \leq \delta.
\end{align}
Hence, $\cone \mathcal{Y}_{out}$ is a $(0, \delta)$-outer approximation of $\cA_{\infty}$ in the sense of Definition~\ref{def_outer_app}.
\end{theorem}
\begin{proof}
If~\eqref{CP} is bounded, Algorithm~\ref{alg_rec} terminates with $\mathcal{Y}_{in} = \mathcal{Y}_{out} = \emptyset$ and the properties are satisfied by definition. Let us consider an unbounded~\eqref{CP}.
Relation~\eqref{eq1} follows from convexity, Lemma~\ref{lemma_alg_rec} and the fact that $\cone \mathcal{Y}_{out} = (\cA_0)_{\infty}$.
Algorithm~\ref{alg_rec} can terminate in two possible ways:
\begin{enumerate}
\item The set $\cY_{in}$ is non-empty (i.e.~a recession direction was found in Step 4) and it holds $\cY_{out} \subseteq \cY_{in} \cup \Delta$.
\item It holds $\diam (\cY_{out}) \leq \delta$.
\end{enumerate}
We consider these two cases separately and show that each one implies~\eqref{eq3}.
\begin{enumerate}
\item In the first case, the termination condition can be stated as
\begin{align}
\label{eq2}
\forall d \in \mathcal{Y}_{out} \quad \exists r \in \mathcal{Y}_{in} : \quad  \Vert d - r \Vert \leq \delta.
\end{align}
We prove~\eqref{eq3} by showing that~\eqref{eq2} implies 
\begin{align}
\label{eq5}
d_H (\cone \mathcal{Y}_{out} \cap B_1, \cone \mathcal{Y}_{in} \cap B_1) \leq \delta.
\end{align}
Given the way the sets are constructed, it holds $\cone \mathcal{Y}_{out} \cap B_1 = \conv (\mathcal{Y}_{out} \cup \{0\})$ and $\cone \mathcal{Y}_{in} \cap B_1 = \conv (\mathcal{Y}_{in} \cup \{0\})$. Take any $d \in \cone \mathcal{Y}_{out} \cap B_1$, then $d = \sum\limits_{i=1}^{k} \alpha^i d^i$ for some $k \in \N, d^1, \dots, d^{k} \in \mathcal{Y}_{out}$ and $\alpha^1, \dots, \alpha^{k} \geq 0$ with $\sum\limits_{i=1}^{k} \alpha^i \leq 1$. According to~\eqref{eq2}, for each $d^i$ there exists $r^i \in \mathcal{Y}_{in}$ with $\Vert d^i - r^i \Vert \leq \delta$. Convexity of the (Manhattan) norm implies $\Vert \sum\limits_{i=1}^{k} \alpha^i r^i \Vert \leq \sum\limits_{i=1}^{k} \alpha^i \Vert r^i \Vert \leq 1$ and $\Vert d - \sum\limits_{i=1}^{k} \alpha^i r^i \Vert \leq \sum\limits_{i=1}^{k} \alpha^i \Vert d^i - r^i \Vert \leq \delta$. This proves~\eqref{eq3} since $\sum\limits_{i=1}^{k} \alpha^i r^i \in \cone \mathcal{Y}_{in} \cap B_1$.

\item Let's consider the second case. 

Convexity of the norm and $\diam (\cY_{out}) \leq \delta$ imply $\diam (\conv \cY_{out}) \leq \delta$. Recall that $\cone \mathcal{Y}_{out} \cap B_1 = \conv (\mathcal{Y}_{out} \cup \{0\}) = \conv (\conv \mathcal{Y}_{out} \cup \{0\})$. Let us now prove~\eqref{eq3}. Take any $y \in \cone \mathcal{Y}_{out} \cap B_1$, then there exists $\bar{y} \in \conv \cY_{out}$ and $\alpha \in [0, 1]$ such that $y = \alpha \bar{y} + (1-\alpha)0$. Since $\cA_{\infty} \cap \conv \cY_{out} \neq \emptyset$ and $\diam (\conv \cY_{out}) \leq \delta$, there exists $\bar{a} \in \cA_{\infty} \cap \conv \cY_{out} \subseteq B_1$ such that $\Vert\bar{y} - \bar{a}\Vert \leq \delta$. Scaling up yields $\Vert y - \alpha\bar{a}\Vert \leq \delta$ for $\alpha\bar{a} \in \cA_{\infty} \cap B_1$, which proves~\eqref{eq3}.

\end{enumerate}
\end{proof}

Now we state an algorithm for solving a (bounded or unbounded) problem~\eqref{CP}. It uses Algorithm~\ref{alg_rec} as an initialization phase and combines it with the iterations of Algorithm~\ref{alg_bound}.

\begin{alg}{\label{alg_unbound}}
Algorithm for solving problem~\eqref{CP}.
\begin{description}
\item[Input] ~\\
A feasible set $\X$ and a matrix $A$; tolerances $\epsilon, \delta > 0$, point $v \in \relint A[\X]$. 

\item[Initialization] ~		
\begin{enumerate}
	\item Run Algorithm~\ref{alg_rec} with tolerance $\delta > 0$ and point $v \in \relint A[\X]$ to obtain a set of feasible points $\bar{X}$ and an outer approximation $\cA_0$ of $\cA$. 
	\item Set $H := \emptyset$ and $V^\epsilon := \emptyset$
\end{enumerate}

\item[Iteration] ~	
\begin{enumerate}\addtocounter{enumi}{2}
	\item While $H \neq \R^a$
	\begin{enumerate}
	\item Set $H := \R^a$
	\item For each $v \in (\text{vert } \mathcal{A}_0)\setminus V^\epsilon$
	\begin{itemize}
	\item Solve problems~\ref{P3} and~\ref{D3} to obtain optimal solutions $(x^v, z^v)$ and  $\lambda^v$. 
	\item Update $\bar{X} := \bar{X} \cup \{x^v\}$
	\item If $\Vert z^v \Vert > \epsilon$, then update $H :=  H \cap \{ y \in \R^a \; : \; \trans{(\lambda^v)} y \geq \trans{(\lambda^v)} Ax^v  \}$. \\ Else $V^\epsilon := V^\epsilon \cup \{v\}$.
	\end{itemize}
	\item Update $\mathcal{A}_0 := \mathcal{A}_0 \cap H$
	\end{enumerate}
\end{enumerate}

\item[Finalization] ~
\begin{enumerate}\addtocounter{enumi}{3}
	\item Compute $\mathcal{Y}_{out} := \text{vert } ((\mathcal{A}_0)_{\infty} \cap B_1  ) \setminus \{0\}$
\end{enumerate}	

\item[Output] ~\\
		Set of feasible points $\bar{X}$ and sets of directions $\cY_{in}, \cY_{out}$.\\
		Outer approximation $\cA_0$ of $\cA$.
\end{description}
\end{alg}

The following lemma and theorem prove the correctness of the algorithm.
\begin{lemma} \label{lemma_alg_unbound}
Assume that $v \in \relint A[\X]$. For Algorithm~\ref{alg_unbound} the following holds:
\begin{enumerate}
\item The set $\cA_0$ is a convex polyhedral superset of $\cA$ throughout the algorithm.
\item If problem~\eqref{CP} is bounded, then Algorithm~\ref{alg_unbound} returns the same sets $\bar{X}$ and $\cA_0$ as Algorithm~\ref{alg_bound}.
\item The recession cone $(\cA_0)_{\infty}$ is a finite $(0, \delta)$-outer approximation of the recession cone $\cA_{\infty}$ throughout the algorithm. 
\end{enumerate}
\end{lemma}
\begin{proof}
\begin{enumerate}
\item Algorithm~\ref{alg_rec} returns a convex polyhedral superset of $\cA$, see Lemma~\ref{lemma_alg_rec}. We only update $\cA_0$ through intersections with halfspaces containing the set $\cA$, see Lemma~\ref{lemma_P3}.
\item For a bounded problem, only Step 2 of Algorithm~\ref{alg_rec} is relevant (no iterations of Steps 4 and 5 are done). Since Step 2 of Algorithm~\ref{alg_rec} coincides with Step 2 of Algorithm~\ref{alg_bound}, the same outer approximation $\cA_0$ and the same set of feasible points $\bar{X}$ are obtained by both. Step 3 of Algorithm~\ref{alg_bound} and Step 3 of Algorithm~\ref{alg_unbound} also coincide, yielding the same sets $\bar{X}$ and $\cA_0$.
\item Within Algorithm~\ref{alg_rec} it holds $\cone \cY_{out} = (\cA_0)_{\infty}$. The claim holds at the end of Step 1 of Algorithm~\ref{alg_unbound} by Theorem~\ref{thm_rec}. Let us denote $\cA_0^1, \cA_0^2, \dots$ the new sets obtained throughout Algorithm~\ref{alg_unbound}. Since these sets are obtained via intersections with supporting halfspaces, it holds $\cA \subseteq \cA_0^i \subseteq \cA_0$ for each $i$. This implies $(\cA)_{\infty} \subseteq (\cA_0^i)_{\infty} \subseteq (\cA_0)_{\infty}$, which proves the claim.
\end{enumerate}
\end{proof}

\begin{theorem}
\label{thm_unbound_alg}
Assume that Algorithm~\ref{alg_unbound} terminated successfully for problem~\eqref{CP}. Then $(\bar{X}, \cY_{out})$ is a finite $(\epsilon, \delta)$-solution of~\eqref{CP}. 
\end{theorem}
\begin{proof}
Let's consider the sets $\bar{X}, \cY_{out}$ and $\cA_0$ upon termination of the algorithm.
According to Lemma~\ref{lemma_alg_unbound}, $\cone \cY_{out} = (\cA_0)_{\infty}$ is a finite $(0, \delta)$-outer approximation of the recession cone $\cA_{\infty}$. Therefore, properties 1 and 2 of Definition~\ref{def_sol_unbound} are satisfied. Lemma~\ref{lemma_alg_unbound} also provides that $\cA_0$ is a convex polyhedron. As such, it has a (non-redundant) V-representation $\cA_0 = \conv \{v^1, \dots, v^k\} + (\cA_0)_{\infty}$, where $k \in \N$. If $(\cA_0)_{\infty}$ contains no lines, then $v^1, \dots, v^k$ are vertices of $\cA_0$. Otherwise these are (non-redundant) points on the boundary of $\cA_0$.  The iterations of Step 3 of Algorithm~\ref{alg_unbound} terminate when for each point $v^i$ there exists $\bar{x}^i \in \bar{X}$ such that $\Vert v^i - A \bar{x}^i \Vert \leq \epsilon$. Therefore, it holds $\conv \{v^1, \dots, v^k\} \subseteq \conv A[\bar{X}] + B_{\epsilon}$. This provides the last property of the solution concept as
$$\cA \subseteq \cA_0 \subseteq \conv A[\bar{X}] + \cone \cY_{out} + B_{\epsilon}.$$
\end{proof}

Algorithm~\ref{alg_rec} provides us with a set of points $\bar{X}$ as well as two sets of directions, $\cY_{in}$ and $\cY_{out}$. It also outputs a polyhedron $\cA_0$. Theorems~\ref{thm_rec} and~\ref{thm_unbound_alg} provide the inclusions
\begin{align} 
\label{eq4}
\conv A[\bar{X}] + \cone \cY_{in} \subseteq \cA \subseteq \cA_0 \subseteq \conv A[\bar{X}] + \cone \cY_{out} + B_{\epsilon}. 
\end{align}
Let us first discuss the left-hand side inclusion. In the following, we will refer to $\conv A[\bar{X}] + \cone \cY_{in}$ as an inner approximation. Inclusion~\eqref{eq4} provides an intuition for this choice of terminology. Let us support this intuition a bit more. In case of a non-empty set $\cY_{in}$, equations~\eqref{eq1} and~\eqref{eq5} imply
\begin{enumerate}
\item $d_H (\cA_{\infty} \cap B_1, \cone \cY_{in} \cap B_1) \leq \delta$,
\item $\cA_{\infty} \supseteq \cone \cY_{in}$, and
\item $\conv A[\bar{X}] + \cone \cY_{in}  \subseteq \cA$.
\end{enumerate}
Note the close resemblance between these properties and those in Definitions~\ref{def_sol_unbound} and~\ref{def_outer_app}.

The set $\cY_{in}$ might, however, be empty for an unbounded problem~\eqref{CP}, see Example~\ref{Example2}. In such a case, the inner approximation $\conv A[\bar{X}] + \cone \cY_{in}$ is a bounded set despite $\cA$ being unbounded. Although such inner approximation is of limited quality, the problem is nevertheless solved in the sense of Definition~\ref{def_sol_unbound} -- as the diameter of the set $\cY_{out}$ is bounded by $\delta$.

Let us now discuss the right-hand side inclusions of~\eqref{eq4}. The polyhedrons $\cA_0$ and $\conv A[\bar{X}] + \cone \cY_{out} + B_{\epsilon}$ are both $(\epsilon, \delta)$-outer approximations of the image set $\cA$, where out of the two outer approximations $\cA_0$ is the better one.
Therefore, if an approximation of the set rather than a solution is of interest, one should chose (depending on the application) between the inner approximation $\conv A[\bar{X}] + \cone \cY_{in}$ and the outer approximation $\cA_0$.

\begin{remark}
\label{Rem:near}
Algorithms~\ref{alg_rec} and~\ref{alg_unbound} assume that the optimal solutions of the scalarizations are found. As we discussed in Remarks~\ref{remark_near1} and~\ref{remark_near2}, in practice only near-optimal solutions will be found. We again denote by $\varepsilon$ the accuracy to which these near-optimal solutions are found. It is reasonable to assume $\varepsilon \ll \min \{\epsilon, \delta\}$. Treating the found near-optimal solutions as optimal solutions (i.e.~omitting $\varepsilon$ within the relevant half-space) means that the half-space used to generate the outer approximation $\cA_0$ are slightly shifted, see Remark~\ref{remark_near1}. However, these half-spaces are not tilted, therefore the recession cone of $\cA_0$ is not affected and we still approximate $\cA_{\infty}$ correctly.
\end{remark}

\begin{remark}
\label{Rem:near}
Note that correctness of Algorithms~\ref{alg_rec} and \ref{alg_unbound} was proven under the assumption that they terminate. Finite termination remains, similarly to Algorithm~\ref{alg_bound}, an open question.
\end{remark}

\section{Examples}
\begin{example}
\label{Example1}
We compute approximations of an intersection of two ellipses projected onto a lower  dimensional space.
Specifcally, we approximate the two-dimensional set
\begin{align}
\label{ex1}
\left\lbrace 
(x_1, x_2) \quad \vert \quad \exists x_3 \in \mathbb{R}: \quad x_1^2 + \frac{(x_2 - 1)^2}{4} + x_3^2 \leq 1, \quad \frac{(x_1 - 1)^2}{4} + x_2^2 + \frac{(x_3 - 1)^2}{4} \leq 1
\right\rbrace
\end{align}
and the three-dimensional set
\begin{align}
\begin{split}
\label{ex2}
\vspace*{-2cm}
\left\lbrace 
(x_1, x_2, x_3) \quad \vert  \quad \exists x_4 \in \mathbb{R}: \quad x_1^2 + \frac{(x_2 - 1)^2}{4} + x_3^2 + \frac{(x_4 - 1)^2}{4} \leq 1, \right. \\
 \left. \quad \frac{(x_1 - 1)^2}{4} + x_2^2 + \frac{(x_3 - 1)^2}{4} + x_4^2 \leq 1
\right\rbrace.
\end{split}
\end{align}
These convex projection problems (or, more precisely, their shifted versions) were solved via the associated multi-objective problems in~\cite{KR21b}. Here we apply Algorithm~\ref{alg_bound}. We compare the original version of the algorithm with the alternative initialization described in Remark~\ref{remark_initialization}, both for a tolerance $\epsilon = 0.01$. For the set~\eqref{ex1}, the alternative initialization leads to fewer computations needed (60 optimizations and 6 evaluations of a polyhedron for the original initialization versus 54 optimizations and 5 evaluations of a polyhedron for the alternative). However, for the set~\eqref{ex2} the opposite is the case (1544 optimizations and 7 evaluations of a polyhedron for the original version and 1570 optimizations and 8 evaluations of a polyhedron for the alternative). Figure~\ref{fig:1} visualizes the obtained approximations.

\begin{figure}[h]
\centering
\caption{\label{fig:1} Approximations of~\eqref{ex1} and~\eqref{ex2} from Example~\ref{Example1} for tolerance $\epsilon = 0.01$.}
\begin{subfigure}[b]{0.45\textwidth}
\centering
\includegraphics[width = \textwidth]{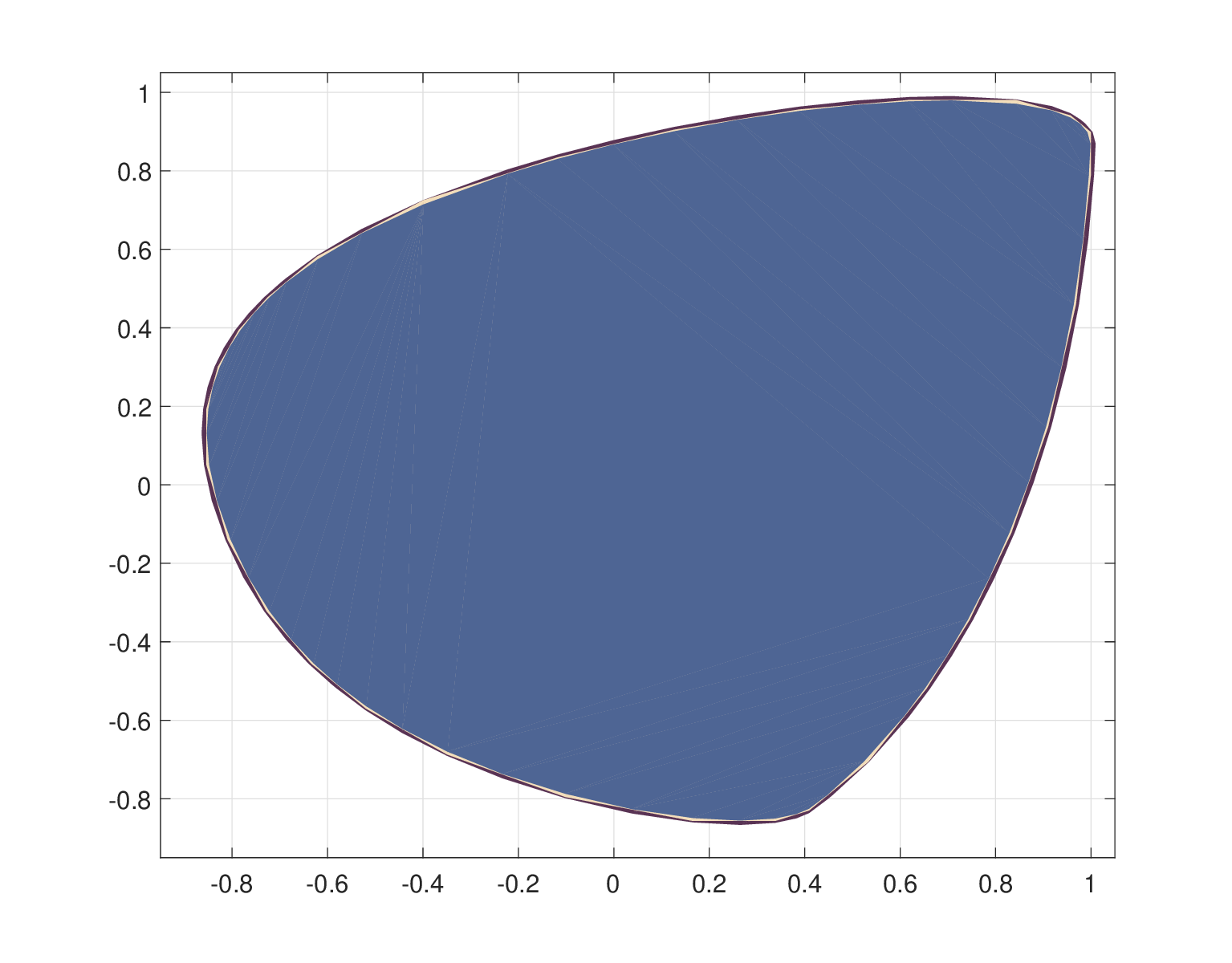}
\caption{\label{fig:1-1}Approximations of~\eqref{ex1}. At this scale the output of different initializations are not visually distinguishable. }
\end{subfigure}
~
\begin{subfigure}[b]{0.45\textwidth}
\centering
\includegraphics[width = \textwidth]{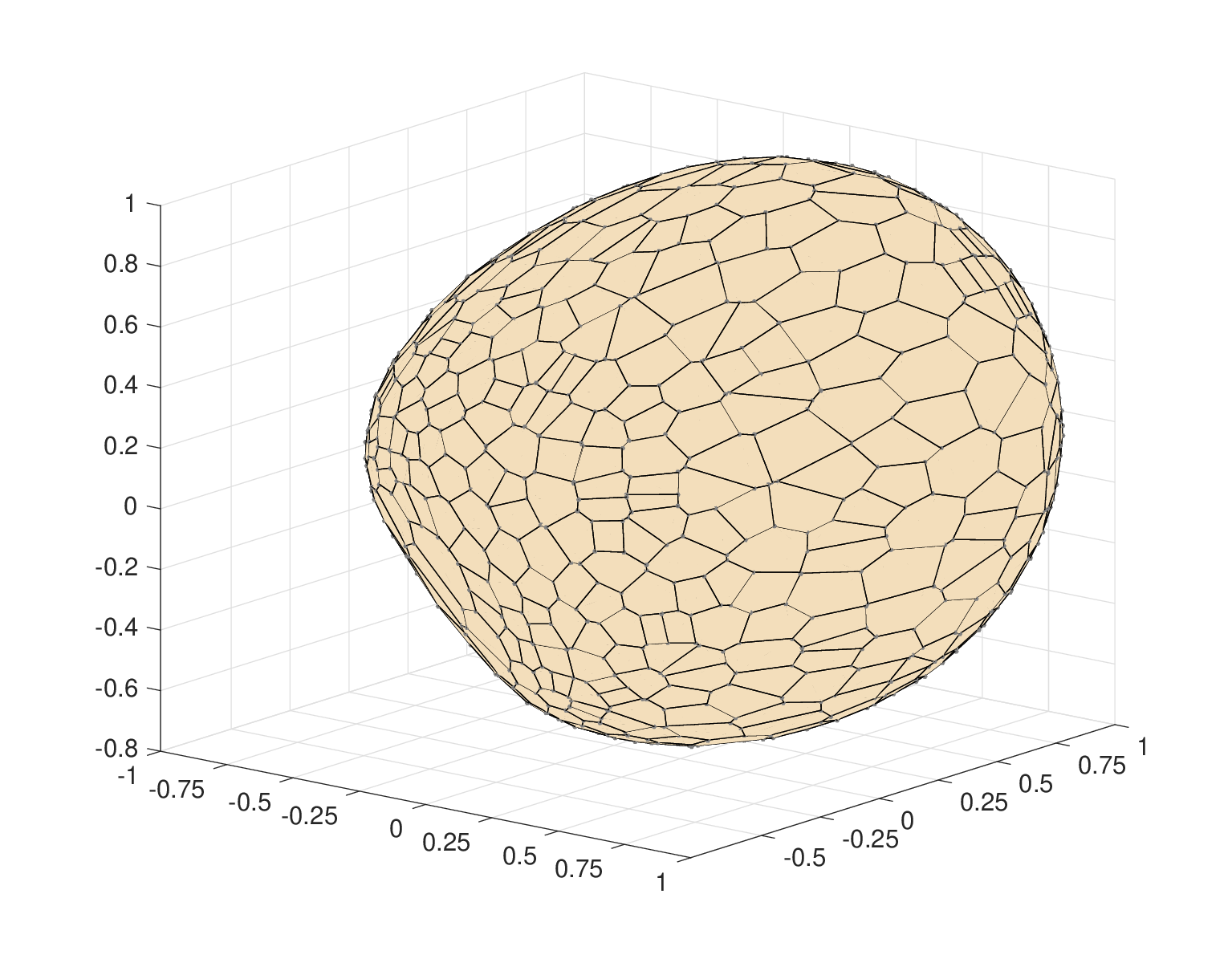}
\caption{\label{fig:1-2}Outer approximation $\cA_0$ of~\eqref{ex2}. }
\end{subfigure}
\end{figure}

\begin{figure}[h]
\centering
\caption{\label{fig:2} Zoomed-in approximations of~\eqref{ex1} from Example~\ref{Example1} for tolerance $\epsilon = 0.01$: Inner approximation $\conv A[\bar{X}]$ in blue, outer approximation $\cA_0$ in yellow and outer approximation $\conv A[\bar{X}] + B_{\epsilon}$ in purple.}
\begin{subfigure}[b]{0.45\textwidth}
\centering
\includegraphics[width = \textwidth]{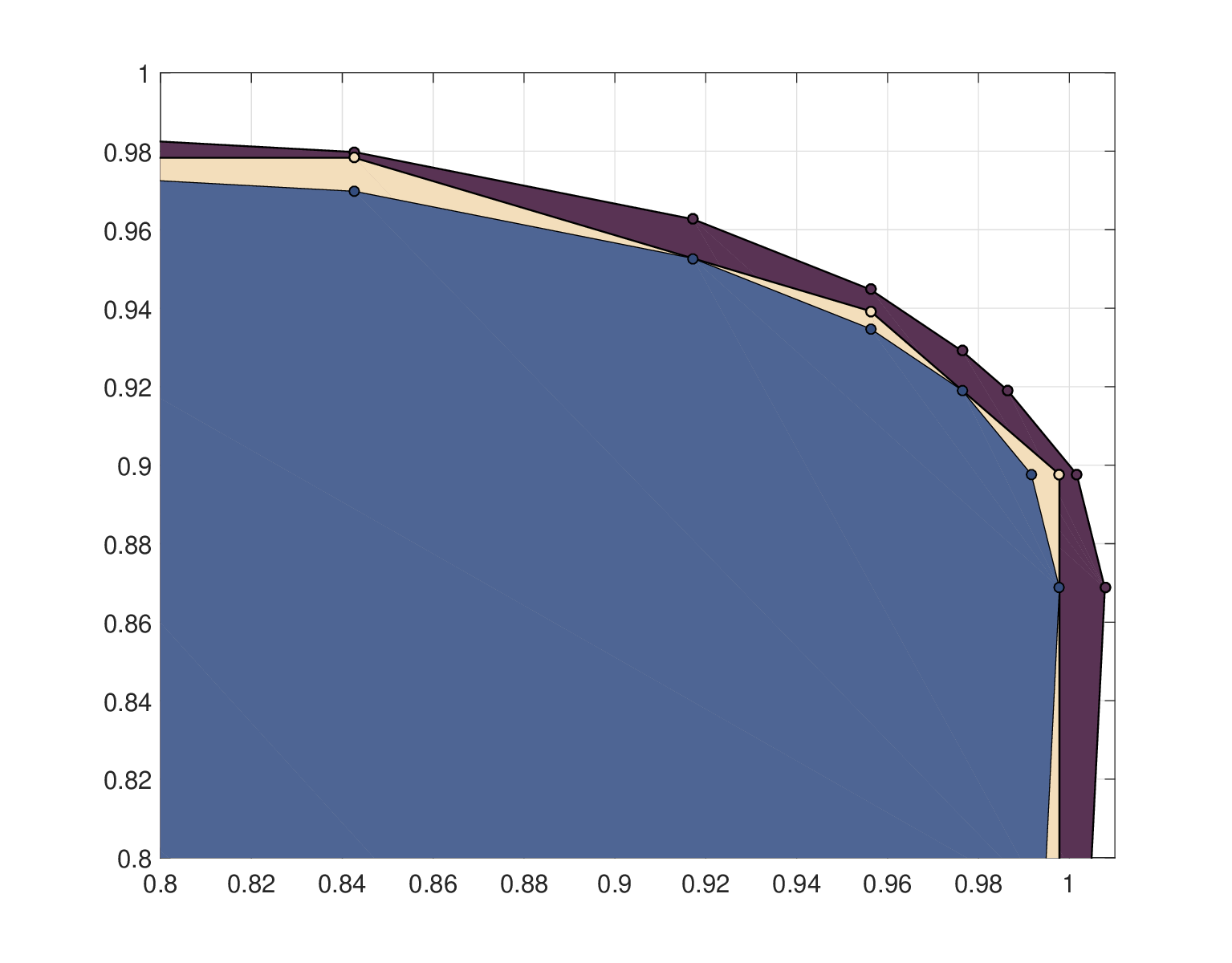}
\caption{Original initialization of Algorithm~\ref{alg_bound}.}
\end{subfigure}
~
\begin{subfigure}[b]{0.45\textwidth}
\centering
\includegraphics[width = \textwidth]{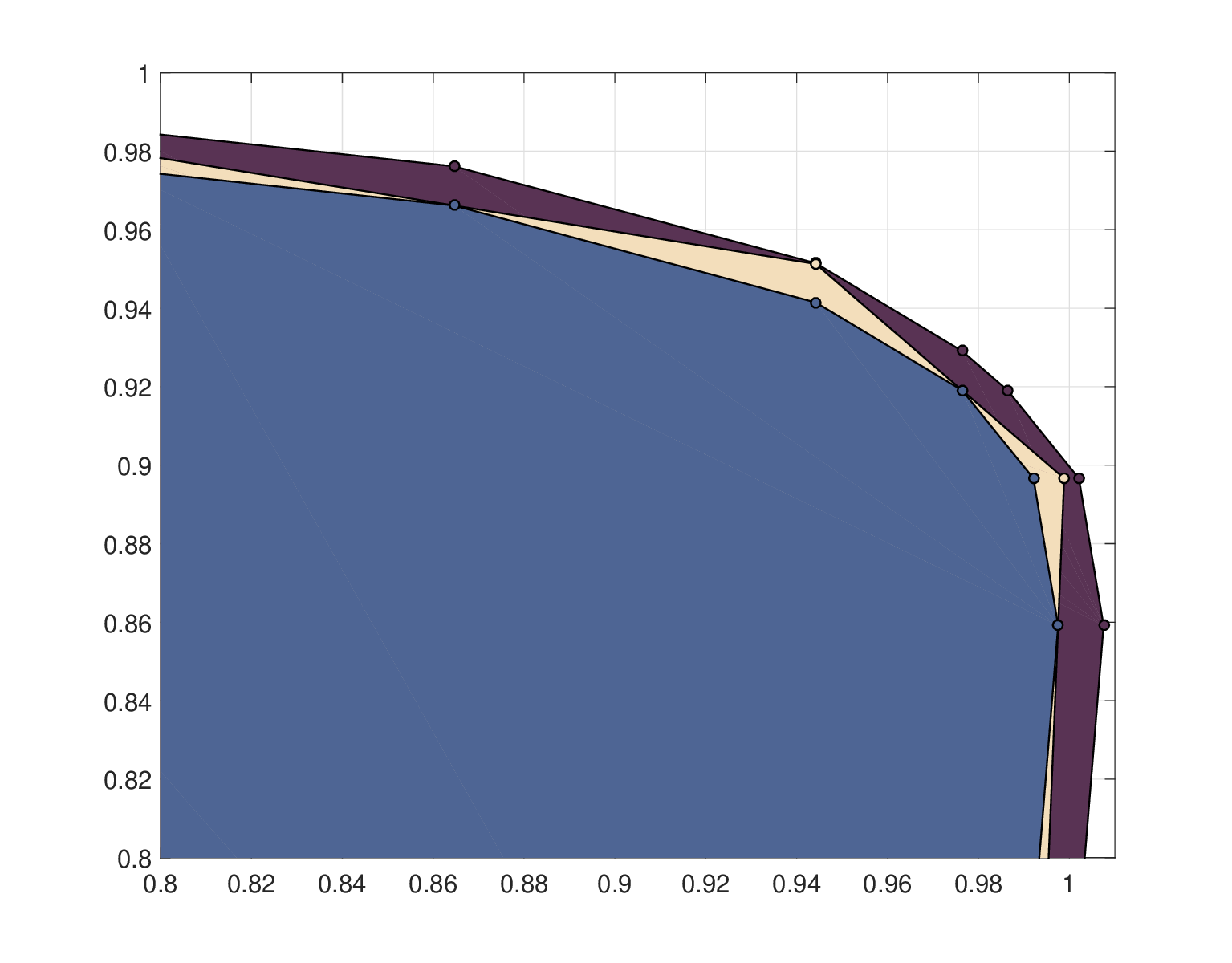}
\caption{Alternative initialization of Algorithm~\ref{alg_bound}.}
\end{subfigure}
\end{figure}

Three approximations for each set are obtained: the inner approximation $\conv A[\bar{X}]$, the outer approximation $\cA_0$ and the outer approximation $\conv A[\bar{X}] + B_{\epsilon}$, see also the discussion around Equation~\eqref{eq4}. All three approximations are depicted in Figure~\ref{fig:1-1}, but they are not visually distinguishable. For the sake of visualization, Figure~\ref{fig:2} contains zoomed-in versions. There we can also observe that the outputs of the algorithm with different initializations differ.
\end{example}

\begin{example}
\label{Example2}
Now we provide an illustrative example of an unbounded problem: We approximate the set 
\begin{align}
\label{ex3}
\left\lbrace 
(x_1, x_2) \quad : \quad (x_1 \cos \theta + x_2 \sin \theta)^2 \leq (x_1 \sin \theta + x_2 \cos \theta)
\right\rbrace
\end{align}
for $\theta = 0$ and $\theta = \frac{\pi}{6}$. These are the quadratic epigraph and its rotated version. We provide Algorithm~\ref{alg_unbound} with a known element of the set as an input. Note that while the theoretical results were derived only for input points from the relative interior of $A[\X]$, one can numerically also experiment with boundary input points. In this example, we input the point $v = (0,2)$ for $\theta = 0$ and $ v = (\cos\theta-\sin\theta, \sin\theta+\cos\theta)$ for $\theta = \frac{\pi}{6}$. For both cases we run the algorithm for tolerance parameters set to $\delta = 0.1$ and $\epsilon = 0.01$.

The obtained approximations for the case of $\theta = 0$ are illustrated in Figure~\ref{fig:3-1}. The algorithm  computes a solution consisting of
\begin{align*}
\cY_{in} = 
\left\lbrace 
\begin{pmatrix}
0 \\ 1
\end{pmatrix}
\right\rbrace , \quad
\cY_{out} = 
\left\lbrace 
\begin{pmatrix}
-0.0327 \\ 0.9673
\end{pmatrix}, 
\begin{pmatrix}
0 \\ 1
\end{pmatrix},
\begin{pmatrix}
0.0327 \\ 0.9673
\end{pmatrix}
\right\rbrace
\end{align*}
as well as $149$ elements of $\bar{X}$. In this case an element of $\cY_{in}$ is found in the first run of Step 5 of Algorithm~\ref{alg_rec}, since the recession direction coincides with a vertex of a unit ball.

This is not the case for the set rotated by $\theta = \frac{\pi}{6}$. No recession direction is found and Step 5  of Algorithm~\ref{alg_rec} terminates when the diameter of $\cY_{out}$ becomes sufficiently small. The algorithm provides a solution of the problem consisting of
\begin{align*}
\cY_{in} = \emptyset, \quad
\cY_{out} = 
\left\lbrace 
\begin{pmatrix}
0.3789 \\ 0.6211
\end{pmatrix}, 
\begin{pmatrix}
0.3459 \\ 0.6541
\end{pmatrix}
\right\rbrace
\end{align*}
and $129$ elements of $\bar{X}$. Note that this means that the inner approximation is bounded despite the set (and the problem) being unbounded. Approximations are displayed in Figure~\ref{fig:3-2}.

\begin{figure}[h]
\centering
\caption{\label{fig:3} Approximations of~\eqref{ex3} from Example~\ref{Example2}: Inner approximation $\conv A[\bar{X}] + \cone \cY_{in}$ in blue and outer approximation $\cA_0$ in yellow. The outer approximation $\conv A[\bar{X}]  + \cone \cY_{out} + B_{\epsilon}$ is not visually distinguishable from $\cA_0$. }
\begin{subfigure}[b]{0.45\textwidth}
\centering
\includegraphics[width = \textwidth]{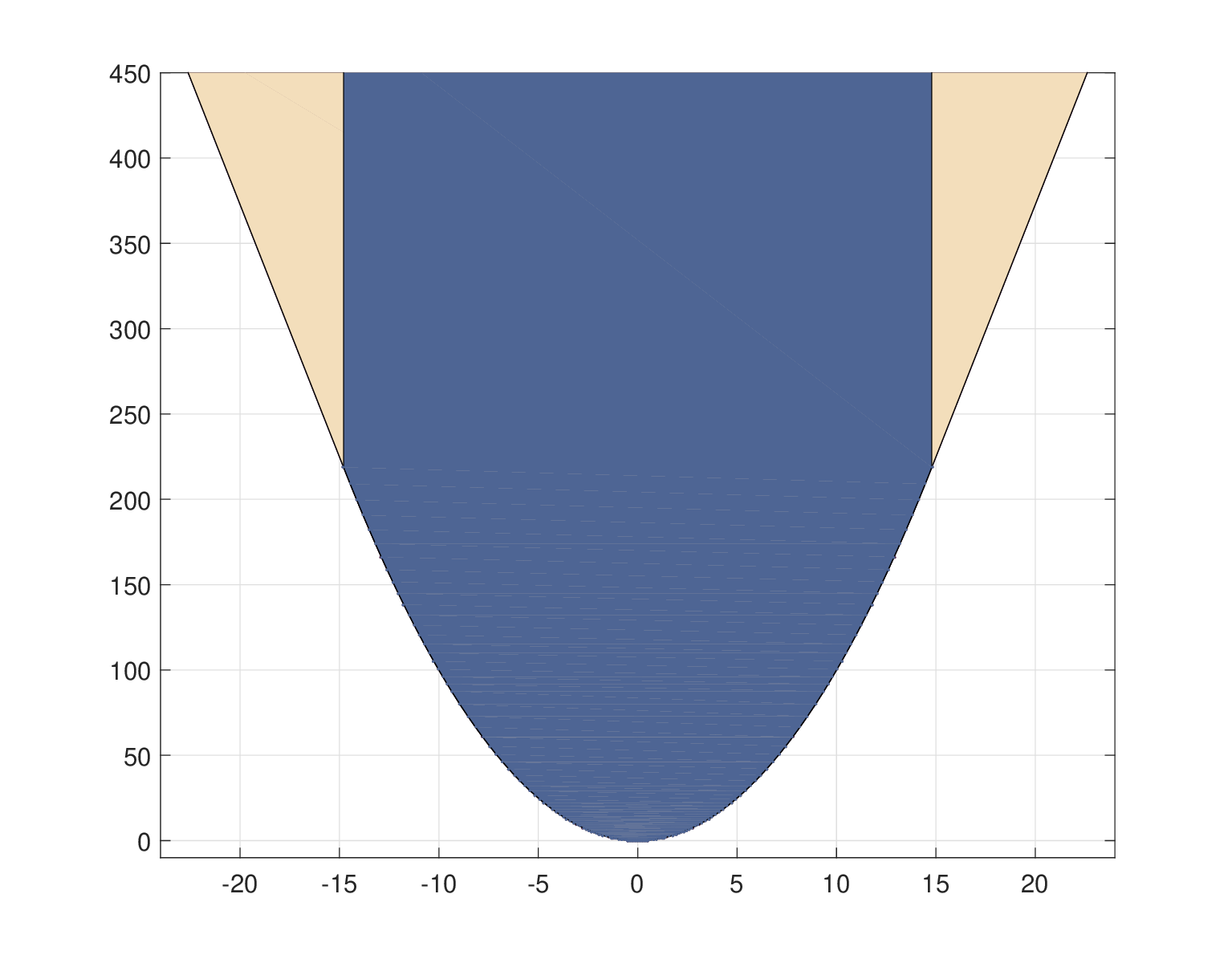}
\caption{\label{fig:3-1} Approximations of the set~\eqref{ex3} for $\theta = 0$. The algorithm was provided with an inner point of the set and it performed 153 optimizations and 13 evaluations of a polyhedron.}
\end{subfigure}
~
\begin{subfigure}[b]{0.45\textwidth}
\centering
\includegraphics[width = \textwidth]{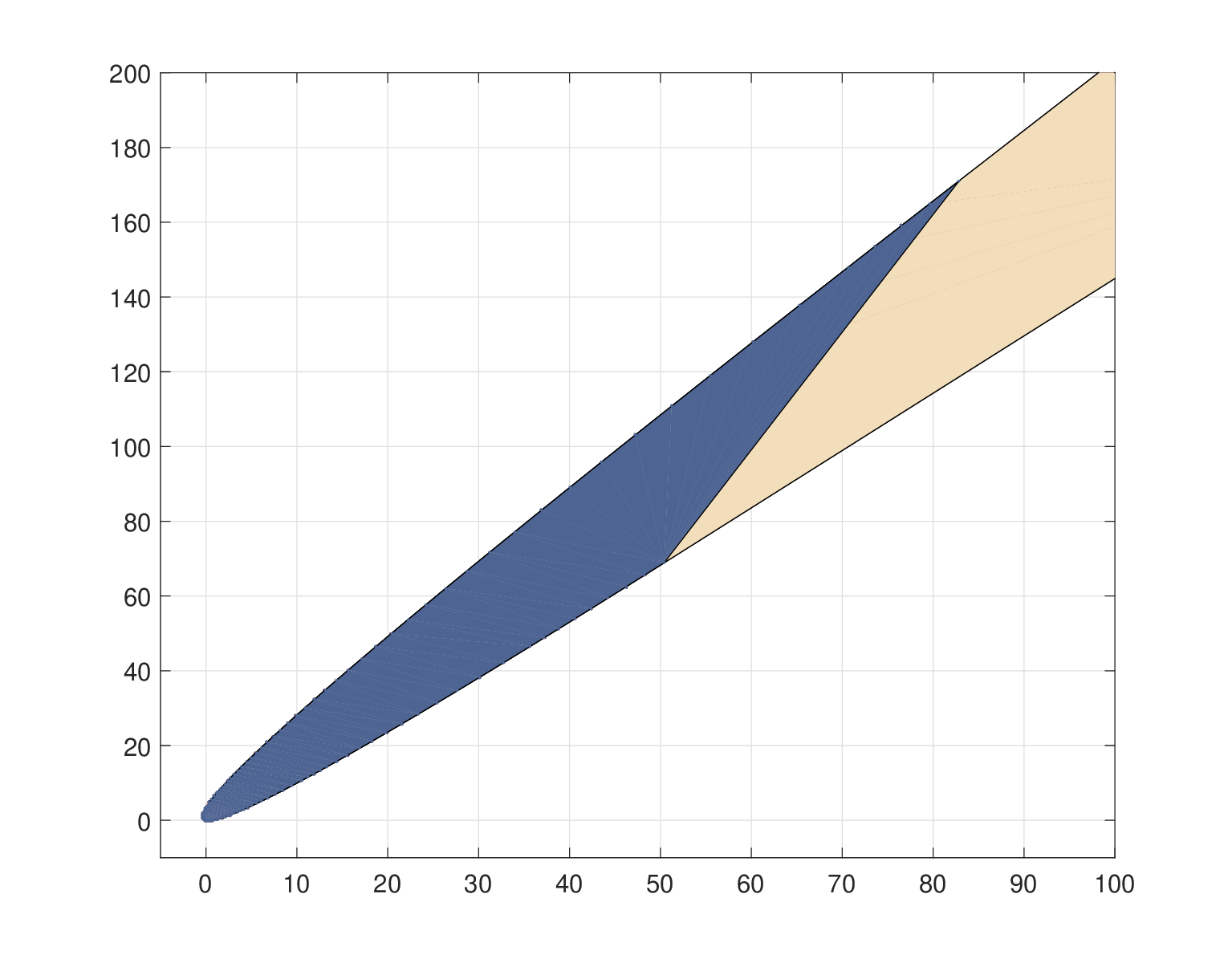}
\caption{\label{fig:3-2} Approximations of the set~\eqref{ex3} for $\theta = \frac{\pi}{6}$. The algorithm was provided with a boundary point of the set and it performed 130 optimizations and 11 evaluations of a polyhedron.}
\end{subfigure}
\end{figure}
\end{example}

\begin{example}
\label{Example3}
We also provide an example of a set with a lineality space. We approximate the unbounded 'tube-like' set
\begin{align}
\label{ex4}
\left\lbrace (x_1, x_2, x_3) \quad : \quad x_1^2 + (x_2 \cos \theta - x_3 \sin \theta)^2 \leq 1\right\rbrace
\end{align}
for $\theta = \frac{\pi}{3}$ via Algorithm~\ref{alg_unbound} for tolerances $\delta = 0.1$ and $\epsilon = 0.01$. In this case we generate the point $v \in A[\X]$ through a feasibility check. Since the three approximations provided by the algorithm are not well visually distinguishable, we only display one of them in Figure~\ref{fig:4-1}.
\end{example}

\begin{example}
\label{Example4}
We revisit the example of approximating the ice cream cone
\begin{align}
\label{ex5}
\left\lbrace (x_1, x_2, x_3) \quad : \quad \sqrt{x_1^2 + x_2^2} \leq x_3 \right\rbrace,
\end{align}
which was solved in~\cite{WURKH22} via a convex vector optimization problem. This time we treat the problem as an instance of~\eqref{CP} and solve it via Algorithm~\ref{alg_unbound} for the same tolerance $\delta = 0.2$ as in~\cite{WURKH22}, which also allows to visually distinguish the three approximations. The approximations are displayed in Figure~\ref{fig:4-2}. 
Note that while the inner approximation $\conv A[\bar{X}] + \cone \cY_{in}$ (in blue) and the outer approximation $\mathcal{A}_0$ (in yellow) are in this case cones, the second outer approximation $\conv A[\bar{X}] + \cone \cY_{out} + B_{\epsilon}$ (in purple) is a shifted cone that does not start in the origin because of the presence of the (polyhedral) ball $B_\epsilon$.
\end{example}

\begin{figure}[h]
\centering
\caption{\label{fig:4} Approximations of~\eqref{ex4} and~\eqref{ex5} from Examples~\ref{Example3} and~\ref{Example4}.}
\begin{subfigure}[b]{0.45\textwidth}
\centering
\includegraphics[width = \textwidth]{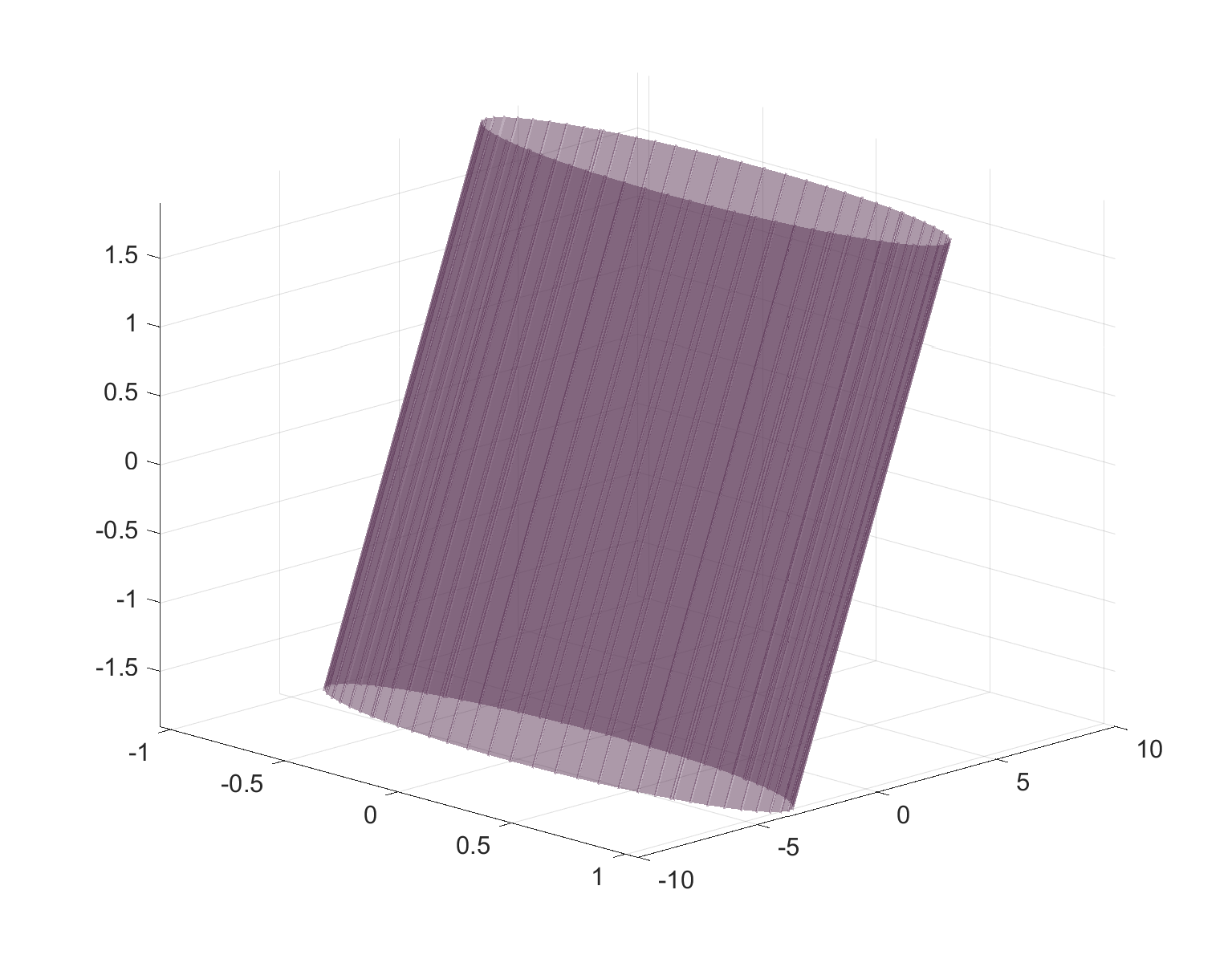}
\caption{\label{fig:4-1} Approximation of~\eqref{ex4}. The algorithm performed 71 optimizations and 7 evaluations of a polyhedron. }
\end{subfigure}
~
\begin{subfigure}[b]{0.45\textwidth}
\centering
\includegraphics[width = \textwidth]{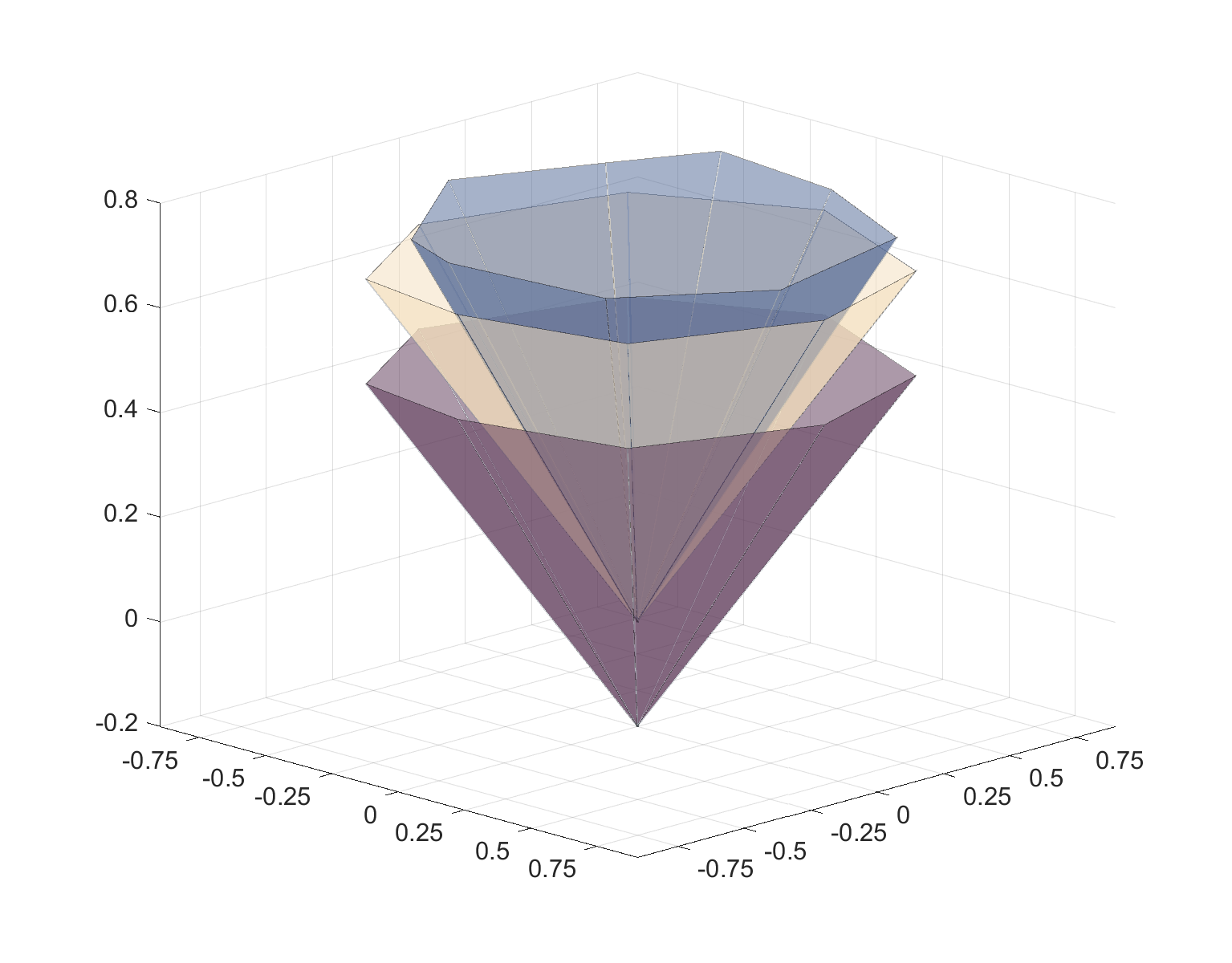}
\caption{\label{fig:4-2} Approximations of~\eqref{ex5}. The algorithm performed 31 optimizations and 8 evaluations of a polyhedron.}
\end{subfigure}
\end{figure}

\section{Conclusions}
The existing literature provides algorithms and methods for the bounded case of a convex projection problem. These are based on (explicitly or implicitly) transforming the convex projection into an associated multi-objective convex problem with one additional dimension of an image space.
 Consequently, the existing methods perform operations  in the higher-dimensional image space of the multi-objective problem.

The first contribution of this paper is Algorithm~\ref{alg_bound}, which solves a bounded convex projection directly, without a transformation to associated multi-objective problem. Algorithm~\ref{alg_bound} is a Benson-type, outer approximation algorithm performing operations in the (lower dimensional) image space of the projection and is inspired by \cite[Algorithm 1]{AUU22}. 
Just like for the outer approximation \cite[Algorithm 1]{AUU22} or \cite[Algorithm 1]{LRU14}, we provide correctness result but finite termination remains an open question.

Algorithm~\ref{alg_bound} serves as a stepping stone to the main result of this paper, a method for solving unbounded convex projection. We introduce an appropriate  solution concept and develop an algorithm computing such solution. Algorithm~\ref{alg_rec} computes  (inner and outer) approximations of a recession cone of a convex projection; it builds on the ideas of \cite[Algorithm 1]{WURKH22} for convex vector optimization problems. Finally, Algorithm~\ref{alg_unbound} combines the ideas of Algorithm~\ref{alg_bound} and Algorithm~\ref{alg_rec} to compute a solution of a convex projection problem. While correctness is proven, finite termination of these (or modified) algorithms remains a question for future research.


\newcommand{\etalchar}[1]{$^{#1}$}

\end{document}